\documentclass[a4paper,12pt]{article}
\usepackage[utf8]{inputenc}
\usepackage[T1]{fontenc}
\usepackage[english]{babel}
\usepackage{amsmath, amssymb, amsfonts, cases, amsthm, mathrsfs}
\usepackage{pifont, lineno, xcolor}
\modulolinenumbers[5]
\usepackage[top=2.5cm,bottom=2.5cm,left=2.5cm,right=2.5cm]{geometry}

\definecolor{couleurliens}{rgb}{.2,0.,.8} % for sections, subsections and equations
\usepackage[pdftex,colorlinks=true,pdfmenubar=true,bookmarks=true,bookmarksopen=true,pdftoolbar=true]{hyperref}
\hypersetup{%
                citecolor=magenta,
                filecolor=blue,
                linkcolor=couleurliens,
                urlcolor=cyan,
                pdftitle={Null hypersurfaces and trapping horizons},  % Titre du document.
                pdfauthor={Hans FOTSING},                      % dans les informations du document
                pdfsubject={Hans Article}           % sous Acrobat.
            }%

\newtheorem{definition}{Definition}[section]
\newtheorem{theorem}{Theorem}
\newtheorem{example}{Example}
\newtheorem{lemma}{Lemma}[section]
\newtheorem{corollary}{Corollary}[section]
\newtheorem{proposition}{Proposition}[section]

\newtheorem{remark}{Remark}[section]

\newcommand{\F}{{\mathcal F}}

\newcommand{\R}{{\mathbb{R}}}

\newcommand{\Bm}{\overline{M}}
\newcommand{\bg}{\overline{g}}
\newcommand{\bnab}{\overline{\nabla}}

\newcommand{\snab}{\stackrel{\star}{\nabla}}
\newcommand{\sn}{\stackrel{\star}{A}_{\xi}}

\newcommand{\sE}{\stackrel{\star}{E}}
\newcommand{\sk}{\stackrel{\star}{k}}
\newcommand{\sS}{\stackrel{\star}{S}}
\newcommand{\st}{\stackrel{\star}{T}}

\title{Null hypersurfaces and trapping horizons}

\author{Hans Fotsing Tetsing\footnote{Corresponding author E-mail:~\href{mailto: hans.fotsing@aims-cameroon.org}{hans.fotsing@aims-cameroon.org}}\\
{\footnotesize Faculty of Sciences, University of Douala\quad and African Institute for Mathematical Sciences}\\
{\footnotesize Po Box: 608 Limbe, Cameroon}\\
Ferdinand Ngakeu\footnote{E-mail:~\href{mailto: fngakeu@yahoo.fr}{fngakeu@yahoo.fr}}\\
{\footnotesize Faculty of Sciences, University of Douala, Po Box: 24157 Douala, Cameroon}}

\begin{document}
\maketitle

\begin{abstract}
The purpose of the present work is to study (marginally) trapped submanifolds lying in a null hypersurface. Let $(M,g,N)\to\Bm(c)$ be a null hypersurface of a space-time with constant sectional curvature $c$, endowed with a Screen Integrable and Conformal rigging $N$. The (Marginally) Trapped Submanifolds we are interested with are particular leaves of the screen distribution according to the sign of their expansions. We prove that if $c$ is non-positive, then $\Bm$ cannot contain a null non-expanding horizon. In the case $c$ is positive, we show that if $\Bm$ satisfies Einstein's equation and dominant energy condition holds, then any null trapping horizon of $\Bm$ is a null non-expanding horizon. More generally we prove that in a spacetime $\Bm(c)$ with constant sectional curvature $c$, cross-sections of a marginally outer trapped tube are Riemann manifold with the same constant sectional curvature $c$.
\end{abstract}

\noindent{\bfseries keywords: }{Monge hypersurface, Null hypersurface, screen distribution, Rigging vector field, Marginally Trapped Submanifold, Trapping Horizon, Black Hole}


\section{Introduction}

Let $(\Bm,\bg)$ be a proper semi-Riemannian manifold and $x:M\to\Bm$ be an embedded hypersurface of $\Bm$. The pull-back metric $g=x^\star\bg$ can be either degenerate or non-degenerate on $M$. When $g$ is non-degenerate, one says that $(M,g)$ is a semi-Riemannian hypersurface of $(\Bm,\bg)$ otherwise $(M,g)$ is called a null (or degenerate, or lightlike) hypersurface of $(\Bm,\bg)$. Since any semi-Riemannian hypersurface has a natural transversal vector field, namely the Gauss map which is everywhere orthogonal to the hypersurface, there is a standard way to study such a hypersurface. Geometrical objects of the ambient manifold $\Bm$ are projected orthogonally on $M$ and give new objects which are used to study the extrinsic geometry of the hypersurface.

On the contrary, for a null hypersurface $M$, the normal bundle $TM^\perp$ is not transversal but rather tangent to $M$. Therefore, other approaches are needed to study the extrinsic geometry of null hypersurfaces. Some authors, for instance, K. L. Duggal, A. Bejancu and M. Gutierrez, B. Olea \cite{DB, GO} have proposed some approaches. In \cite{DB}, it is proved that for every choice of a supplementary distribution $S(TM)$ (called a screen distribution) of $TM^\perp$ in $TM$, and for every choice of a null section $\xi$ of $TM^\perp$, there exists a unique rank one bundle $tr(TM)$ transverse to $M$ and a (locally defined) null section $N$ of $tr(TM)$ such that $\bg(\xi,N)=1$. To study the extrinsic geometry of $M$, geometrical objects of $\Bm$ are projected on $M$ parallelly to $tr(TM)$. Some difficulty with this method is a large number of arbitrary elections need and the fact that induced objects are locally defined. 

In \cite{GO}, the authors consider a vector field $\zeta$ defined on an open subset containing $M$ and everywhere transversal to $M$ (called a rigging vector field). This vector $\zeta$ fixes a unique transversal bundle and a unique screen distribution. Notice that, a rigging vector field may not exists for a given null hypersurface. However for a spacetime (time-orientable Lorentzian manifold) $\Bm$, there exists a timelike vector field globally defined on $\Bm$. This timelike vector field is rigging for any null hypersurface of $\Bm$, since a timelike vector field can't be tangent to a null hypersurface in a Lorentzian ambient.  

Introduced by Penrose in \cite{PR}, the concept of \emph{trapped surfaces} plays an important role in general relativity. A spacelike surface $S$ is said to be a trapped surface if all light rays emitted from the surface locally converge. Nothing can escape, not even the light. It is believed that there will be a marginally trapped surface separating the trapped surfaces from the untrapped ones where the outgoing light rays are instantaneously parallel. For example in stationary spacetimes, the event horizon of a black hole is a marginally trapped surface.

Galileo's principle according to which all bodies fall equally fast is the equivalent to the Newtonian principle saying that the initial mass (the $m$ in the fundamental Newton formula $F=ma$) and the passive gravitational mass (the mass acted on by a gravitational field) are equal for a given body \cite{DP}. Hence for these two theories, gravity is a field present in the universe and which affects all bodies. In general relativity, the gravitational field is the manifestation of the curvature of the spacetime which is the consequence of the presence of the matter and
%Since any Lorentzian manifold is locally conformally flat (isothermal coordinate system), (1) relativity fail for small scales and (2) freely falling observers fell no gravitational field at sufficiently small scales (this is the relativist equivalent of the above-mentioned Galileo and Newton principles). Hence, 
no notion of an intense gravitational field can be attached to one single spacetime point: a local notion becomes necessary. A normal bundle of a spacelike surface $S$ can be spanned by two future-directed null vector fields, say $k$ and $\ell$. (We set $\ell$ to be in the outgoing direction.) 
Since trajectories of light are null geodesics, $S$ can be taken as an initial event for sending two pulses of light: one toward one side of the surface (say inward) and the other toward the other side (say outward). When the gravitational field is weak, the pulse of light sent outward will increase its area, while the pulse of light sent inward will have decreasing area. If the gravitational field near the surface is intense and directed inward, it is possible that the outward light geodesics may bend inward sufficiently so that the area of the light fronts decreases. This geometric fact is taken as an indicator of the presence of a strong gravitational field. Spacelike surfaces where this behaviour occurs are called trapped surfaces and the ones with a behaviour borderline between the "normal" situation and the strong gravitational field situation are called marginally trapped. \cite[For more physical comment on (marginally) trapped surfaces.]{Ja, LMW, Ja2, GJ, Gou}.

Let $x:S\to(\Bm,\bg)$ be a spacelike codimension two submanifold of an $(n+2)-$dimensional space-time $\Bm$. Then there exist two future directed null vector field $k$ and $\ell$ spanning the normal bundle $TS^\perp$ and normalized as $\bg(k,\ell)=-1$. The expansions of $S$ with respect to $k$ and $\ell$ are respectively defined as the traces $\theta^{(k)}$ and $\theta^{(\ell)}$ of $\bnab k$ and $\bnab\ell$ with respect to the induced metric $x^*\bg$, being $\bnab$ the Levi-Civita connexion of $\bg$. the submanifold $S$ is called a trapped submanifold when the expansions $\theta^{(k)}$ and $\theta^{(\ell)}$ have different signs, and is called marginally outer trapped submanifold when at least one of these expansions vanishes. 

Let $x:M\to\Bm$ be a null hypersurface of a $(n+2)-$dimensional spacetime $(\Bm,\langle,\rangle)$. Let $N$ be a null vector field defined on $\Bm$ and everywhere transversal to $M$: we call $N$ a rigging vector field for $M$. Let $\xi$ be the unique null vector field on $M$ such that $\langle N,\xi\rangle=1$ ($\xi$ is called the associated rigged vector field) and $S(N)=ker(x^\star\langle N,\cdot\rangle)$ the associated screen distribution. The shape operator $A_N$ (resp. $\sn$) is defined on sections of $TM$ as the projection of the covariant derivative $-\bnab_\cdot N$ (resp. $-\bnab_\cdot\xi$) on the screen distribution along $N$ (resp. $\xi$). The null hypersurface $M$ (resp. the screen distribution) is said to be totally umbilical if there is a function $\rho$ (resp. $\lambda$) such that $\sn=\rho P$ (resp. $A_N=\lambda P$), $P$ being the projection from $TM$ onto $S(N)$ along $\xi$. The vector field $N$ is said to be conformal if there exists a function $\varphi$ such that $A_N=\varphi\sn$. A Screen Integrable and Conformal (SIC) rigging is a conformal rigging with integrable associated screen distribution. Let's assume that $N$ is a SIC rigging and call $L$ a generic leaf of the screen distribution. Then $L$ is a codimension two spacelike submanifold of $\Bm$ and its normal bundle $TL^\perp$ is spanned by $k=N$ and $\ell=-\xi$, assume to be future directed. Expansions of $L$ are given by $\theta^{(k)}=-tr(A_N)|_L$ and $\theta^{(\ell)}=tr(\sn)|_L$. Hence when $M$ is minimal (i.e trace of $\sn~$ vanishes), leaves of $S(N)$ are marginally outer trapped submanifolds and $M$ is called a marginally outer trapped tube and leaves of $S(N)$ are called cross-sections. 

This paper is organized as follows. This section is labelled Introduction. In the next Section \ref{section2}, we recall general setup and notations on null hypersurfaces. Particularly, we mention some useful results on rigged lightlike hypersurfaces and the behaviour of some geometrical objects under the change of rigging. Section \ref{section3} is devoted to the main results of this paper. We characterize (marginally) trapped submanifold lying in a null hypersurface $M$ by geometrical objects of $M$. For instance, we prove that cross-sections of a null marginally outer trapped tube in a spacetime with constant sectional curvature $c$, have the same constant sectional curvature $c$. We also prove that if $M$ is totally umbilical with $\sn=\rho P$, then leaves of the screen distribution are space forms with constant sectional curvature $\kappa=c+2\varphi\rho^2$. We investigate the case of Monge null hypersurfaces in Lorentz-Minkowski space $\R^{n+2}_1$ and find that they cannot be foliated by trapped submanifold. 
%We also obtain the following results.\vspace{.2cm}
%
%\noindent {\bfseries Theorem A}:~~A null hypersurface $(M,g)$ of a semi-Riemannian space form $(\Bm(c),\bg)$ is totally geodesic if and only if it is minimal.\vspace{.2cm}
%
%\noindent {\bfseries Theorem B}:~~Let $M$ be a null hypersurface of a semi-Riemannian manifold satisfying the null convergence condition. Then, $M$ is totally geodesic if and only if $M$ is minimal.

\section{General setup on null hypersurfaces}\label{section2}

Throughout this work, $(\Bm, \bg)$ is a $(n+2)-$dimensional $(n\geq1)$ semi-Riemannian manifold of index $q\geq1$. (From Section \ref{section3}, the couple $(\Bm,\bg)$ is a Lorentzian manifold time-orientable (and then space-orientable) of signature $(-,+,\dotsb,+)$, that is a spacetime.) $\bnab$ and $\bar{R}$ will denote respectively the Levi-Civita connection and the Riemannian curvature of $\bg$. The metric $\bg$ may be denoted by $\langle,\rangle$. All manifolds are assumed smooth and connected. Let $\Sigma$ be a $d-$dimensional manifold with $d\leq n+2$. If there exists an immersion $x:\Sigma\to\Bm$ then, $x(\Sigma)$ is called a $d-$dimensional {\bfseries immersed submanifold} of $\Bm$.  If moreover $x$ is injective one says that $x(\Sigma)$ is a $d-$dimensional {\bfseries submanifold} of $\Bm$. If in addition the inverse map $x^{-1}$ is a continue map from $x(\Sigma)$ to $\Sigma$, $x(\Sigma)$ is a $d-$dimensional {\bfseries embedded submanifold} of $\Bm$. When $x(\Sigma)$ is an embedded submanifold, one identifies $\Sigma$ and $x(\Sigma)$. All submanifolds will be taken embedded and through the identification state before, saying that $x:M\to\Bm$ is a submanifold will mean that there is an embedding $x:\Sigma\to\Bm$ such that $M=x(\Sigma)$. An hypersurface of $\Bm$ is a submanifold of $\Bm$ of dimension $d=n+1$. We will said that $x:(M,g)\to(\Bm,\bg)$ is an isometrically immersed submanifold when, $x:M\to\Bm$ is a submanifold of $\Bm$ and $g=x^\star\bg$. An isometrically immersed submanifold $x:(M,g)\to(\Bm,\bg)$ is called a {\bfseries spacelike submanifold} (resp. a {\bfseries Lorentzian submanifold}) if $(M,g)$ is a spacelike (resp. a Lorentzian) manifold, and a {\bfseries null submanifold} (or a {\bfseries lightlike submanifold}) if the metric $g$ is degenerate. The latter means that at each point $p\in M$ there exists a nonzero vector $u\in T_pM$ such that $g_p(u,v)=0$ for any $v\in T_pM$.

Let $x:(M,g)\to(\Bm,\bg)$ be an isometrically immersed lightlike hypersurface. Then, the normal bundle $TM^\bot$ is a sub-distribution of $TM$ and there is noway to define orthogonal projection from $T_p\Bm$ into $T_pM$. To study the geometry of the null hypersurface $M$, we need to choose smoothly a transverse direction to $T_pM$ and define up to this transverse a projection from $T_p\Bm$ into $T_pM$. There are infinitely many possibilities to choose a complementary of $TM^\bot$ in $TM$ and such a complementary is called {\bfseries screen distribution}. Let $\mathcal S(TM)$ be a fixed screen distribution. One has,
\begin{equation}\label{decomp0}
TM=\mathcal S(TM)\oplus_{orth}TM^\bot.
\end{equation}
Firstly we have:
\begin{theorem}[\cite{DB}]\label{db} 
	Let $(M,g,S(TM))$ be a null hypersurface endowed with a chosen screen distribution $\mathcal S(TM)$. Then, there exists a unique rank $1$ vector bundle $tr(TM)$ over $M$, such that for any nonzero section $\xi$ of $TM^\bot$ on a coordinate neighborhood $\mathcal U\subset M$, there exists a unique section $N$ of $tr(TM)$ on $\mathcal U$ satisfying
	\begin{equation}\label{normalization}
	\bg(\xi, N)=1, \;\; \bg(N,N)=\bg(N,W)=0, \;\;\forall
	W\in\Gamma(\mathcal S(TM)|_{\mathcal{U}}).
	\end{equation}
\end{theorem}
One then has the decomposition,
\begin{equation}\label{decomp}
T\Bm_{|M}=TM\oplus tr(TM)=\mathcal
S(TM)\oplus_{orth}\left(TM^\bot\oplus tr(TM)\right).
\end{equation}

Secondly, it is noteworthy that the choice of a null
transversal vector field $N$ along $M$ determines the null transversal vector bundle, the screen distribution $S(TM)$ and a unique radical vector field, $\xi$ say, satisfying (\ref{normalization}) and (\ref{decomp}).  Now, to continue our
discussion, we need to clarify the concept of rigging for our null
hypersurface. In \cite{GO}, the authors introduce the notion of
rigging for null hypersurfaces of Lorentzian manifolds. In the present work, we show
that this notion can be generally defined when the ambient
space is a semi-Riemannian manifold as well as it is the case here.

\begin{definition}\label{rigging1}
	A {\bfseries rigging}  for $M$ is a vector field $\zeta$ defined over $M$ such that
	$\displaystyle \zeta_{p}\notin T_{p}M$ for each $p\in M$.
\end{definition}

A {\bfseries null rigging} for $M$ is a rigging $\zeta$ such that for all $p\in M$, $\bg_p(\zeta_p,\zeta_p)=0$. Let $\zeta$ be a rigging for $M$, $\eta=\bg(\zeta,\cdot)$
and $\widetilde g=g+\eta\otimes\eta$.

\begin{lemma}\label{tildeg}
	$\widetilde g$ is a non-degenerate metric on $M$.
\end{lemma}

\begin{proof}
	Let $u\in T_pM$ such that $\widetilde g_p(u,v)=0$ for every $v\in
	T_pM$. In particular, for $\xi\in TM^\perp$, one has $0=\widetilde
	g_p(u,\xi_p)=\bg_p(\xi_p, \zeta_p)\bg_p(u,\zeta_p)$. Since
	$\xi\in\Gamma(TM^\perp)$ and $\bg$ is non-degenerate,
	$\bg_p(\xi_p,\zeta_p)\neq0$ and then $\bg_p(u,\zeta_p)=0$. gathering this
	with the fact that $T_p\Bm_{|M}=span\{\zeta_p\}\oplus T_pM$, one has
	$\bg_p(u,v)=0$ for every $v\in T_p\Bm$, which implies that $u=0$,
	since $\bg_p$ is a non-degenerate metric.
\end{proof}

The triple $(M,g,\zeta)$ is called a {\bfseries rigged null hypersurface} and $\widetilde g$ the {\bfseries associated metric}, it is a semi-Riemannian metric of index $q-1$. One defines the screen distribution associated to the chosen rigging $L$ by $S(TM)=ker(\eta)$.

\begin{definition}
	The {\bfseries rigged} vector field associated to $\zeta$ is the vector field $\widetilde
	g-$metrically equivalent to the $1-$form $\eta$. It is denoted by $\xi$.
\end{definition}

\begin{lemma}
	Let $x:(M, g, \zeta)\to(\Bm,\bg)$ be a rigged null hypersurface, $\xi$ the associated rigged vector field and $S(TM)$ the associated screen distribution. Then, $N=\zeta-\frac12\bg(\zeta,\zeta)\xi$ is a (restricted) null rigging and equations in (\ref{normalization}) and decompositions in (\ref{decomp}) hold with transverse bundle $tr(TM)=span\{N\}$. 
\end{lemma}

\begin{proof}
	It is easy to see that $N$ is a null rigging and $\eta=\bg(N,\cdot)$. By definition, $\widetilde g(\xi,\xi)=g(\xi,\xi)+\eta(\xi)^2=\bg(N,\xi)^2\neq0$. In another hand since $\xi$ is the $\widetilde g-$metrically equivalent vector to the $1-$form $\eta$, one has $\widetilde g(\xi,\xi)=\eta(\xi)=\bg(N,\xi)$. It follows that $\bg(N,\xi)=1$. Since $S(TM)=ker(\eta)$, the other equality in (\ref{normalization}) holds and decompositions in (\ref{decomp}) are straightforward.
\end{proof}

As from any rigging, one has a (restricted) null rigging, in what follows, we will consider a rigged null hypersurface $x:(M,g,N)\to(\Bm,\bg)$ endowed with a null rigging $N$. An advantage of rigging is that the null transverse $N$ is globally defined on $M$. But the existence of rigging is not always granted for any null hypersurface.  However if $(\Bm,\bg)$ is a (time-orientable) spacetime then, there exists a timelike vector field $\zeta$ globally defined on $\Bm$. This timelike vector field is rigging for any null hypersurface of $\Bm$ (since a timelike vector field can't belong to the tangent space of a null hypersurface of a Lorentzian manifold) and one can use $\zeta$ to derive a null rigging as above. We say that $(M, g, N)$ is a closed rigged null hypersurface when the $1-$form $\eta$ is closed.                                                                                                                    

%\section{Normalization on a null hypersurface}

\begin{lemma}
	Any closed rigged null hypersurface $(M,g,N)$ is foliated.
\end{lemma}

\begin{proof}
	Since the associated screen distribution $S(TM)$ is the kernel of the closed $1-$form $\eta$, then by the Frobenius theorem, the screen distribution $S(TM)$ is integrable and its leaves foliate $M$.
\end{proof}

Let $\nabla$ be the connection on $M$ induced from $\bnab$ through the projection along the transverse bundle $tr(TM)$. For every section $U$ of $TM$, one has $\bg(\bnab_U\xi,\xi)=0$. The Weingarten map is the endomorphism field
$$\begin{matrix}
\chi:&\Gamma(TM)&\to&\Gamma(TM)\\&U&\mapsto&\bnab_U\xi
\end{matrix}.$$

The Gauss-Weingarten equations of the immersion $x:M\to\Bm$ are given by
\begin{eqnarray}
\bnab_UV&=&\nabla_UV+B(U,V)N,\label{geq1}\\
\nabla_UPV&=&\stackrel\star\nabla_UPV+C(U,PV)\xi,\label{geq2}\\
\bnab_UN&=&-A_NU+\tau(U)N, \label{geq3}\\
\nabla_U\xi&=&-\sn U-\tau(U)\xi, \label{geq4}
\end{eqnarray}
for all $U,V\in\Gamma(TM)$, where $\stackrel\star\nabla$, denotes the connection on the screen distribution $S(TM)$ induced from $\nabla$ through the projection morphism $P$ of $\Gamma(TM)$ onto $\Gamma(S(TM))$ with respect to the decomposition (\ref{decomp0}). $B$ and $C$ are the local second fundamental forms of $M$ and $\mathcal S(TM)$ respectively, $A_N$ and $\stackrel\star A_\xi$ are the shape operators on $TM$ and $S(TM)$ respectively, and the {\bfseries rotation $1-$form} $\tau$ is given by
\begin{eqnarray*}
	\tau(U)=\bg(\bnab_UN,\xi).
\end{eqnarray*}

\begin{remark}
	Note that in general relativity when $M$ is an isolated horizon of a dimension $4$ black hole $\Bm$, the two first $\Psi_0$ and $\Psi_1$ of Weyl components vanish on $M$ and the only probably non-vanishing component $\Psi_2$ is gauge invariant and related to  $\tau$ by $-d\tau=(Im\Psi_2)\varepsilon$, being $\varepsilon$ the natural area $2-$form on $M$ and $Im\Psi_2$ the imaginary part of $\Psi_2$. When $Im\Psi_2$ vanishes, all angular momentum multipoles vanish and this horizon is said to be non-rotating, see \cite{ashtekar2004isolated,ashtekar2004multipole} for more details. This explains why $-\tau$ is called rotation $1-$form of an isolated horizon. However, $-\tau$ is called by H\'a\'jicek \cite{Ha,Ha3} \emph{a gravimagnetic field} and by Damour \cite{Da,Da2} \emph{a surface momentum density}. 
\end{remark}

It is easy to check that the Weingarten map and the second fundamental form of $M$ are related by
\begin{equation}
B(U,V) =-g(\chi(U),V),\quad  \forall U,V \in~\Gamma(TM),
\end{equation}
which show that the Weingarten map is $g-$symmetric. Also, second fundamental forms and sharp operators are related by
\begin{equation}\label{autoadjoint}
B(U,V) =g(\sn\!U,V),\quad  C(U,PV) =g(A_{N}U, V)\quad  \forall U,V \in~\Gamma(TM).
\end{equation}
One checks that
\begin{equation}\label{deg}
B(U,\xi) = 0, \quad \sn\xi = 0~~\forall U \in~\Gamma(TM).
\end{equation}

\begin{remark}
	Null hypersurface $M$ is ruled by integral curves of $\xi$. It follows from \eqref{geq4} and (\ref{deg}) that integral curves of $\xi$ are pregeodesics in both $\Bm$ and $M$, as $\displaystyle \bnab_{\xi}\xi = -\tau^{N}(\xi)\xi$. Hence $\xi$ remains collinear to itself when it is parallely transported along its integral curves, thus one can find a function $\alpha$ such that the rescaling $\alpha\xi$ is geodesic. When $\tau(\xi)\neq0$, for any integral curve $\gamma:s\mapsto\gamma(s)$, $s$ is not an affine parameter of the geometric curve plot of $\gamma$. For this reason, $\tau(\xi)$ is called the non-affinity coefficient. Also, integral curves of $\xi$ are called {\bfseries null generators} of $M$. When $M$ is the horizon of a Kerr black hole, $-\tau(\xi)$ is called the {\bfseries surface gravity} of the black hole, see \cite{GJ} for more details. We will refer to $\tau(\xi)$ as the surface gravity.
\end{remark}

Let $\widetilde{N}$ be another rigging for $M$. There exists a section $\zeta$ of $TM$ and a nowhere vanishing smooth function $\phi$ such that $\widetilde{N}=\phi N+U_0$. The following Lemma gives relationships between geometrical objects described above. 

\begin{lemma}[\cite{C3}]\label{change}
	Let $N$ and $\widetilde{N}$ be two riggings for $M$ with $\widetilde{N}=\phi N+U_0$, where $U_0\in\Gamma(TM)$ and $\phi\in\mathcal C(M)$. Then,
	\begin{dingautolist}{192}
		\item $\widetilde\xi=\frac1\phi\xi$ ;
		\item $2\phi\eta(U_0)+\langle U_0,U_0\rangle=0$ ;
		\item $B^{\widetilde N}=\frac1\phi B^N$ ;
		\item $\widetilde P=P-\frac1\phi g(U_0,\cdot)\xi$ ;
		\item $\widetilde\nabla=\nabla-\frac1\phi B(\cdot,\cdot)U_0$ ;
		\item $\tau^{\widetilde N}=\tau^N+d\ln|\phi|+\frac1\phi B(U_0,\cdot)$ ;
		\item $\stackrel{\star}{A}_{\widetilde\xi}=\frac1\phi\sn-\frac{1}{\phi^2}B^N(U_0,\cdot)\xi$ ;
		\item $A_{\widetilde N}=\phi A_N-\nabla_\cdot U_0+[\tau^N+d\ln|\phi|+\frac1\phi B(U_0,\cdot)]U_0$.
	\end{dingautolist}
\end{lemma}

Assume that $\xi$ is defined on an open subset of $\Bm$ containing $M$. Using (\ref{deg}), one gets that for vector fields $U,V$ tangent to the null hypersurface, one has $(L_\xi)(U,V)\bg=(L_\xi g)(U,V)$. Hence, if $\xi$ is a killing vector field for $\Bm$ then, $\xi$ is a killing vector field for $M$. Using Gauss-Weingarten equations, it is nothing to check that
\begin{equation}\label{lxi}
L_\xi g=2B.
\end{equation}

\begin{definition}[\cite{compere2006introduction}]
	A {\bfseries killing horizon} is a rigged null hypersurface $(M,g,\xi)$ of a spacetime $(\Bm,\bg)$, whose the rigged $\xi$ is defined on an open subset of $\Bm$ containing $M$ and is a killing vector field in $\Bm$.
\end{definition}
The rigidity theorem of black holes proved firstly by Hawking states that: in any stationary and asymptotically flat spacetime with a black hole, the event horizon is a Killing horizon \cite{hawking1973large}.

A null hypersurface $M$ is said to be \textit{totally umbilical}
(resp. \textit{totally geodesic}) if there exists a smooth
function $\rho$ on $M$ such that at each $p\in M$ and for all $u,v
\in T_{p}M$, $B^{N}(p)(u,v) = \rho(p) g(u,v)$  (resp. $B^{N}$
vanishes identically on $M$). These are intrinsic notions on any
null hypersurface in the way that they do not depend on the choice of the rigging (see Lemma \ref{change} item \ding{194}). The total umbilicity and the total geodesibility conditions for $M$ can also be written
respectively as $\sn ~=\rho P$ and $\sn ~= 0$. Also, we say that the screen
distribution $S(TM)$ is \textit{totally umbilical} (resp.
\textit{totally geodesic}) if $\displaystyle C^{N}(U,PV) = \lambda
g(U,V)$ for all $U,V\in \Gamma(TM)$ (resp. $C^{N} = 0$), which is
equivalent to $A_{N} = \lambda P$ (resp. $A_{N}= 0$). It is
noteworthy to mention that the shape operators $\sn~$ and $A_{N}$
are $S(TM)-$valued. From equation \eqref{lxi}, one deduce the following.

%\begin{proposition}
%A null hypersurface $M$ of spacetime is a killing horizon if and only if $M$ is totally geodesic.
%\end{proposition}

One can also say that a rigged null hypersurface is a killing horizon if and only if the Weingarten map is $TM^\perp-$valued.

\begin{lemma}\cite{DB}\label{dbumbilical}
	Let $(M,g,N)$ be a totally umbilical rigged null
	hypersurface of an $(n+2)-$dimensional pseudo-Riemannian
	space-form. Then $\rho$ from the above definition satisfies
	\begin{eqnarray}
	\xi(\rho)+\rho\tau^N(\xi)-\rho^2&=&0\label{dbumbilical1}\\
	PU(\rho)+\rho\tau^N(PU)&=&0,
	\end{eqnarray}
	for all $U\in \Gamma(TM)$.
\end{lemma}

%Let denote by $R$ the Riemannian curvature tensor of $\nabla$. The following are Gauss-Codazzi equations
%\begin{eqnarray}
% \bg(\bar R(U,V)PW,N)&=&g((\nabla_UA_N)V,PW)-g((\nabla_VA_N)U,PW)\nonumber\\
%&+&\tau(V)g(A_NU,PW)-\tau(U)g(A_NV,PW)\label{geq5}\\
%\bg(\overline R(U,V)W,\xi)&=&g((\nabla_U\sn)V,W)- g((\nabla_V\sn)U,W)\nonumber\\
%&+&\tau(U)g(\sn V,W)-\tau(V)g(\sn V,W).\label{geq6}
%\end{eqnarray}

The induced connection $\nabla$ is torsion-free but may not be $g-$metric unless $M$ is totally geodesic. In fact we have for all sections $U, V, W$ of $TM$,
\begin{equation}\label{met}
(\nabla_Ug)(V,W)=\eta(V)B(U,W)+\eta(W)B(U,V).
\end{equation}

One defines the mean curvature $\stackrel{\star}{S}_1$ of the null hypersurface $M$ and the mean curvature $S_1$ of the screen distribution $S(TM)$ as the trace of the endomorphism fields $\sn$ and $A_N$ respectively.  Thus,
$$ \stackrel\star S_1=tr(\sn) \mbox{ and } S_1=tr(A_N).$$

\begin{definition}\label{minimal}
	The null hypersurface (resp. the screen distribution $S(TM)$) is said {\bfseries minimal} if $\stackrel\star S_1$ (resp. $S_1$) identically vanishes.
\end{definition}

\begin{lemma}
	For every vector field $U\in\Gamma(TM)$, one has
	\begin{equation}
	tr(\nabla_U\!\!\sn)=U\cdot\!\stackrel{\star}{S}_1
	\end{equation}
\end{lemma}

\begin{proof}
	Since $\sn$ is diagonalizable, there exists a quasi-orthonormal frame $\{\sE_0=\xi,\sE_1,\ldots,\sE_n\}$ of eigenvectors with corresponding real eigenfunctions $\stackrel{\star}{k}_0=0,\stackrel{\star}{k}_1,\ldots,\stackrel{\star}{k}_n$ such that
	\begin{equation*}
	\langle \sE_i, \sE_j\rangle=\epsilon_i\delta_{ij}, \mbox{ with } \epsilon_i=\pm1,~~ \forall i,j=1,\ldots,n.
	\end{equation*} 
	It follows that
	$$tr(\nabla_U\!\!\!\sn)=\sum_{i=1}^{n}\epsilon_i\langle\left(\nabla_U\!\!\sn\right)\sE_i,\sE_i\rangle+\langle\left(\nabla_U\!\!\sn\right)\xi,N\rangle=\sum_{i=1}^{n}U\cdot\!\sk_i=U\cdot\!\sS_1.$$
\end{proof}

Gauss-Codazzi equations of $(M,g,N)$ are given for $U,V,W\in\Gamma(TM)$ and $X\in\Gamma(S(TM))$ by
\begin{align}
\langle\overline{R}(U,V)W,X\rangle&=\langle R(U,V)W,X\rangle\nonumber\\\label{ge1}&+B(U,W)C(V,X)-B(V,W)C(U,X)\\
\langle\overline{R}(U,V)W,N\rangle&=\langle R(U,V)W,N\rangle\\
\langle\overline{R}(U,V)X,N\rangle&=\left(\nabla_UC\right)(V,X)-\left(\nabla_VC\right)(U,X)\nonumber\\&+C(U,X)\tau(V)-C(V,X)\tau(U),\label{ge2}\\
\langle\overline{R}(U,V)W,\xi\rangle&=\left(\nabla_UB\right)(V,W)-\left(\nabla_VB\right)(U,W)\nonumber\\&+B(V,W)\tau(U)-B(U,W)\tau(V).\label{ge3}\\
\langle\overline{R}(U,V)\xi,N\rangle&=C(\sn U,PV)-C(U,\sn\! PV)-d\tau(U,V),\label{ge4}
\end{align}
where $R$ is the Riemann curvature of $\nabla$. As definition of the Ricci tensor of $\overline{R}$ we use 
\begin{equation*}
\overline{Ric}(\overline{U}, \overline{V})=tr\left(\overline{W}\mapsto\overline{R}(\overline{W},\overline{U})\overline{V}\right),
\end{equation*}
for all $\overline{U}, \overline{V}\in\Gamma(\overline{M})$.

\begin{lemma}
	The following equation holds
	\begin{equation}\label{ric}
	\overline{Ric}(\xi,\xi)=\xi\cdot\!\sS_1+\tau(\xi)\!\sS_1-tr\!\left(\sn^2\right).
	\end{equation}
\end{lemma}
\begin{proof}
	Let $\st_1:=-\sS_1\!\!I+\sn:\Gamma(TM)\to\Gamma(TM)$ be the first Newton transformation of $\sn$, as defined in \cite{AF}. By the Lemma \ref{tildeg}, $\widetilde{g}$ is a non-degenerate metric on $M$ and it is easy to check that the precedent frame $\{\sE_0=\xi,\sE_1,\ldots,\sE_n\}$ is an $\widetilde{g}-$orthonormal frame. Hence, 
	$$div^\nabla\left(\st_1\right)=tr\left(\nabla\st_1\right)=\sum_{\alpha=0}^{n}\epsilon_\alpha\left(\nabla_{\sE_\alpha}\!\!\!\st_1\right)\sE_\alpha=\sum_{\alpha=0}^{n}\epsilon_\alpha\left[(\nabla_{\sE_\alpha}\!\!\!\sn)\sE_\alpha-\sE_\alpha(\sS_1)\sE_\alpha\right].$$
	For $U\in\Gamma(TM)$, using covariant derivative formula (\ref{met}) one has 
	\begin{align*}
	\big\langle div^\nabla\left(\st_1\right),U\big\rangle&=\big\langle\sum_{\alpha=0}^{n}\epsilon_\alpha\nabla_{\sE_\alpha}\!\!\!\sn\sE_\alpha-\sn\nabla_{\sE_\alpha}\!\!\!\sE_\alpha,U\big\rangle-\sum_{\alpha=0}^{n}\epsilon_\alpha\sE_\alpha(\sS_1)g(\sE_\alpha,U)\\
	&=\sum_{\alpha=0}^{n}\epsilon_\alpha\langle\sE_\alpha,(\nabla_{\sE_\alpha}\!\!\!\sn)U\rangle-tr\left(\sn^2\right)\eta(U)-PU(\sS_1).
	\end{align*}
	Using Gauss-Codazzi equation (\ref{ge3}), one has $$\langle\sE_\alpha,(\nabla_{\sE_\alpha}\!\!\!\sn)U\rangle=\langle\sE_\alpha,(\nabla_{U}\!\!\!\sn)\sE_\alpha\rangle+\langle\overline{R}(\sE_\alpha,U)\sE_\alpha,\xi\rangle+B(\sE_\alpha,\sE_\alpha)\tau(U)-B(\sE_\alpha,U)\tau(\sE_\alpha).$$ 
	Above relation becomes 
	\begin{align*}
	\big\langle div^\nabla\left(\st_1\right),U\big\rangle&=tr\left(\nabla_{U}\!\!\!\sn\right)+\sum_{\alpha=0}^{n}\epsilon_\alpha\left[\langle\overline{R}(\sE_\alpha,U)\sE_\alpha,\xi\rangle-B(\sE_\alpha,U)\tau(\sE_\alpha)\right]\\&+\tau(U)\sS_1-tr\left(\sn^2\right)\eta(U)-PU(\sS_1)\\
	\big\langle div^\nabla\left(\st_1\right),U\big\rangle&=tr\left(\nabla_{U}\!\!\!\sn\right)-\overline{Ric}(U,\xi)-\tau\left(\sn\!\! U\right)+\tau(U)\sS_1-tr\left(\sn^2\right)\eta(U)-PU(\sS_1).
	\end{align*}
	Using the above Lemma and taking $U=\xi$, one obtains (\ref{ric}).
\end{proof}
Equation (\ref{ric}) is called null Raychaudhuri's equation. Since $\sn$ is diagonalizable, and $\overline{Ric}(\xi,\xi)=0$ when sectional curvature is constant, the following result holds.
\begin{theorem}[\cite{AF}]\label{geomini}
	A null hypersurface $(M,g)$ of a semi-Riemannian manifold $(\Bm(c),\bg)$ with constant sectional curvature $c$ is totally geodesic if and only if it is minimal.
\end{theorem}

Prove of this theorem is straightforward from null Raychaudhuri's equation \eqref{ric}. A consequence of this result is that if a semi-Riemannian manifold $\Bm$ admits a null hypersurface which is minimal but not totally geodesic then, $\Bm$ does not have constant sectional curvature. For the end of this section, we assume that $(\Bm,\bg)$ is a spacetime and $\xi$ is past-directed. If one assumes in addition that Einstein's equation $\overline{Ric}-\frac12\overline{\mathbf{R}}\bg=8\pi\mathbf{T}$ holds, then In Raychaudhuri's equation one can replace the Ricci tensor by the energy-momentum tensor to get
\begin{equation}
8\pi\mathbf{T}(\xi,\xi)=\xi\cdot\!\sS_1+\tau(\xi)\!\sS_1-tr\!\left(\sn^2\right).
\end{equation} 
From this equation, one sees that if the null energy condition holds (thus $\mathbf{T}(v,v)\geq0$ for any null vector field $v$), then $M$ is totally geodesic if and only if $M$ is minimal. This is a well-known result \cite{GJ}. A variant of this theorem is the null splitting theorem of Gregory J. Galloway \cite{gregory2000} where the null energy condition is taken as $\overline{Ric}(v,v)\geq0$. 

The outgoing and ingoing null expansions of a rigged null hypersurface $(M,g,N)$ of a spacetime $(\Bm,\bg)$ with respect the directions $\ell=-\xi$ and $k=N$ are respectively given by
\begin{equation}
\theta^{(\ell)}=tr(\sn)\qquad\mbox{and}\qquad\theta^{(k)}=-tr(A_N)
\end{equation}

\begin{definition}[\cite{krishnan2014quasi}]
	A {\bfseries Non-Expanding Horizon} (NEH) in a $4-$dimensional space-time $(\Bm,\bg)$ is a rigged null hypersurface $(M,g,N)$ such that:
	\begin{itemize}
		\item $M$ has topology $\mathbb{S}^2\times\R$;
		\item the ingoing null expansion $\theta^{(\ell)}$ vanishes on $M$;
		\item Einstein's equation hold, and the matter stress-energy tensor $\mathbf{T}_{ab}$ is such
		that for any future directed null-normal $-\mathbf{T}^a_b\ell^b$ is future causal (dominant energy condition).
	\end{itemize}
\end{definition}
The first condition means there exists an homeomorphism $f:\mathbb{S}^2\times\R\to M\subset\Bm$, such that for any $x\in\mathbb{S}^2$, the image of the map $t\mapsto f_x(t)=f(x,t)$ is a null curve in $\Bm$ (in fact, an integral line of $\xi$), and immersions $x\to f_t(x)=f(t,x)$ form a foliation of $M$ by submanifolds orthogonal to $\xi$ and homeomorphic to the sphere, and the distribution corresponding to this foliation is a screen distribution.

%\begin{theorem}
%Let $M$ be a null hypersurface of a semi-Riemannian manifold satisfying the null convergence condition. Then, $M$ is totally geodesic if and only if $M$ is minimal.
%\end{theorem}

\begin{definition}
	A manifold $M$ endowed with a torsion-free linear connection $\nabla$ is said to be a {\bfseries locally flat} manifold if the Riemannian curvature of $\nabla$ identically vanishes.
\end{definition}

\section{Marginally (Outer) Trapped Submanifolds}\label{section3}

In what follows, $(\Bm,\bg)$ is an $(n+2)-$dimensional spacetime, i.e a time-orientable Lorentzian manifold.

Let $S$ be a spacelike codimension-two submanifold of the spacetime $\Bm$. (Some authors call $S$ a surface.) Let $\ell$ and $k$ be two future-directed lightlike vector fields of $\Bm$ normalized by $\langle \ell,k\rangle=-1$ and such that $TS^\perp=span\{\ell,k\}$. We assume that $\ell$ is in the outgoing direction and $k$ ingoing direction. For all sections $X,Y$ of the tangent space $TS$, Gauss and Weingarten formulas of the immersion $S\to\Bm$ are given by

\begin{align}
\bnab_XY&=\snab_XY+\Pi(X,Y)\label{geq1s}\\
\bnab_Xk&=-A^+_{k}X+\snab^\perp_Xk\label{geq2s}\\
\bnab_X\xi^+&=-A^+_{\ell}X+\snab^\perp_X\ell,\label{geq3s}
\end{align}
where  $\snab$ is the Levi-Civita connection of $S$, the symmetric tensor $\Pi$ is the second fundamental form, $A^+_{k},A^+_{\ell}:\Gamma(TS)\to\Gamma(TS)$ are the shape operators with respect to $k$ and $\ell$ respectively.

The expansions of $S$ with respect to $k$ and $\ell$ are defined as the traces $\theta^{(k)}=-tr(A^+_{k})$, $\theta^{(\ell)}=-tr(A^+_{\ell})$ respectively. The mean curvature vector is given by $H=tr(\Pi)=\theta^{(k)}\ell+\theta^{(\ell)}k$. Let $N$ be a compactly supported normal vector to $S$ and $(\phi^N_\epsilon)_{\epsilon\in I}$ the associated one parameter group of diffeomorphisms of $\Bm$. For each $\epsilon$, $S_\epsilon=\phi^N_\epsilon(S)$ is called a {\bfseries Lie dragging} of $S$ along $N$. Hence $S_0=S$ and $\phi^N_\epsilon$ can be viewed as a variation of the immersion $\phi^N_0:S\to\Bm$ of $S$ into $\Bm$ with the velocity vector of the variation $N=\frac{d\phi^N_\epsilon}{d\epsilon}_{|\epsilon=0}$. Let $|S_\epsilon|$ denote the area of $S_\epsilon$. From the first variation formula one gets
\begin{equation}\label{area}
\frac{d}{d\epsilon}|S_\epsilon|_{|\epsilon=0}=-\int_{S}\langle H,N\rangle\eta_g,
\end{equation} 
where $\eta_g$ is the metric form on $S$. The first order variation of the area of $S$ with respect to deformations along $N$ is $\delta_N|S|:=\frac{d}{d\epsilon}\left[|S_\epsilon|\right]_{|\epsilon=0}$. It follows that 
$$\delta_k|S|=\int_{S}\theta^{(k)}\eta_g\qquad\mbox{ and }\qquad\delta_\ell|S|=\int_{S}\theta^{(\ell)}\eta_g.$$ Hence when $\theta^{(\ell)}<0$, the area of $S$ decrease when $S$ is dragging along $\ell$; this is taken as a clear signal of the presence of a strong gravitational field which that tends to drag things and even light toward it \cite{hayward2000black}, and $S$ is called a {\bfseries weakly future trapped surface} \cite{mars2014stability}.  

Introduced by Penrose in \cite{PR}, the concept of \emph{trapped surface} plays an important role in general relativity. The surface $S$ is said to be a trapped surface if all light rays emitted from the surface locally converge. Nothing can escape, not even light. It is believed that there will be a marginally trapped surface, separating trapped surfaces from the untrapped ones, where the outgoing light rays are instantaneously parallel. It is prove that, $S$ is a trapped surface if and only if the two expansions are of the same sign, and marginally trapped if and only if (at least) one of the expansions vanishes, which is equivalent to say that the mean curvature vector field is lightlike or zero The following definitions can be found in \cite{Ja2,GJ,Ja2,alan2013black}.
\begin{definition}\label{margtrap}
	A codimension-two spacelike submanifold of a Lorentzian manifold is called a future
	\begin{itemize}
		\item {\bfseries Trapped Submanifold (TS)} if $\theta^{(\ell)}<0$ and $\theta^{(k)}<0$.
		\item {\bfseries Marginally Trapped Submanifold (MTS)} if $\theta^{(\ell)}=0$ and $\theta^{(k)}\leq0$.
		\item {\bfseries Trapped Outer Submanifold (TOS)} if $\theta^{(\ell)}<0$.
		\item {\bfseries Marginally Outer Trapped Submanifold (MOTS)} if its mean curvature vector is lightlike or zero.
	\end{itemize}
\end{definition}

The outgoing direction depends on the choice and when the expansion in one direction is zero one takes this direction as the outgoing direction. In other words, a MOTS is a codimension-two spacelike submanifold for which the expansion in a normal null direction vanishes.

\begin{definition}[\cite{hayward1994general}]
	\begin{itemize}
		\item The {\bfseries trapped region} of the Lorentzian manifold $\Bm$ is an inextensible region for which each point lies on some trapped submanifold $S\subset\Bm$. Its boundary is called a {\bfseries trapping boundary}.
		\item A {\bfseries trapping horizon} of $\Bm$ is (the closure of) a hypersurface $M$ foliated by closed MOTSs, called cross sections, for which $\theta^{(k)}\neq0$ and $\delta_{k}\theta^{(\ell)}\neq0$.
		\item A trapping horizon $M$ is said future (respectively, past) if for each MOTS leaf of $M$, there exists $k$ and $\ell$ (as above) such that  $\theta^{(k)}<0$ (respectively, $\theta^{(k)}>0$).
		\item A trapping horizon $M$ is said outer (respectively, inner) if for each MOTS leaf of $M$, there exists $k$ and $\ell$ (as above) such that  $\delta_{k}\theta^{(\ell)}<0$ (respectively, $\delta_{k}\theta^{(\ell)}>0$).
	\end{itemize}
\end{definition}

These notions of trapped surfaces are used to capture a practical idea of the back hole which can be simulated numerically. More precisely, one sees locally a black hole as a Future Outer Trapping Horizon (FOTH). In such a region of the space-time, the ingoing light rays are converging ($\theta^{(k)}<0$) and outgoing light rays are  instantaneously parallel on the horizon ($\theta^{(\ell)}=0$), and their converging inside the horizon ($\delta_{k}\theta^{(\ell)}<0$). A {\bfseries Marginally Outer Trapped Tube} (MOTT) is a hypersurface foliated by MOTSs \cite{andersson2007stability}. When the MOTSs foliating a MOTT are closed and the inward null expansion $\theta^{(k)}$ and the variation of the outward null expansion $\theta^{(\ell)}$ along the inward direction do not vanish, the MOTT becomes a trapping horizon. If in addition, we are in a $4-$dimensional space-time where Einstein's equation and dominant energy condition hold, then we obtain an NEH.
%Recall that for any covariant tensor $\Gamma$ on $S$, the first order variation of $\Gamma$ along a vector field $N$ is defined by $\delta_N(\Gamma)=\frac{d}{d\epsilon}\left((\phi^N_\epsilon)^\star(\Gamma_\epsilon)\right)|_{\epsilon=0}$, where $\Gamma_\epsilon$ is the analogous of $\Gamma$ on $S_\epsilon=\phi^N_\epsilon(S)$.
%\begin{definition}
%A {\bfseries quasi-local black hole horizon} is a Future Outer Trapping Horizon (FOTH).
%\end{definition}

\subsection{Null Trapping Horizon}

Let $x:(M,g,N)\to(\Bm,\bg)$ be a rigged null hypersurface of a spacetime $(\Bm,\bg)$, with $N$ future-directed. (Recall that since $\Bm$ is time-orientable, any null hypersurface $M\to\Bm$, has a null rigging vector field.) Let's assume that the rigging $N$ is with an integrable screen distribution $S(TM)=ker(x^\star\langle N,\cdot\rangle)$ and we denote by $S$ a generic leaf. It is a well known fact that $S$ is a codimension-two spacelike submanifold of  $(\Bm,\bg)$ and the normal bundle $TS^\perp$ of the immersion $S\to\Bm$ is spanned by $N$ and $\xi$. We set $k=N$ and $\ell=-\xi$ to be the future-directed null normals spanning the normal bundle of $S$, with $\ell$ in the outgoing direction. Using equations (\ref{geq1}) to (\ref{geq4}), we derive Gauss and Weingarten formulas of the immersion   $S\to\Bm$ as 
\begin{align}
\bnab_XY&=\snab_XY+B(X,Y)N+C(X,Y)\xi=\snab_XY+h(X,Y),\label{geq7}\\
\bnab_XN&=-A_NX+\tau(X)N=-A_NX+\snab^\perp_XN, \label{geq8}\\
\nabla_X\xi&=-\sn X-\tau(X)\xi=-\sn X+\snab^\perp_X\xi, \label{geq9}
\end{align}
for any $X,Y\in\Gamma(S(TM))$, where $h(X,Y)$ is the second fundamental form, and $\snab^\perp$ is the normal connection. Hence the screen connection $\snab$ is the Levi-Civita connexion of $S$, and the     shape operators of the null hypersurface $M$ and the ones of the immersion $S\to\Bm$ (defined by (\ref{geq2s}) and (\ref{geq3s})) are related on $S$ by 
$$A_N=A^+_{k}~~\mbox{ and }~~\sn =-A^+_{\ell}.$$
Hence, expansions of $S$ and mean curvatures of $M$ are related by
\begin{equation}\label{expansion}
\begin{matrix}
\theta^{(k)}&=&-tr(A^+_{k})=-tr(A_N)=-S_1\\
\theta^{\ell}&=&-tr(A^+_{\ell})=tr(\sn)=\sS_1
\end{matrix}.
\end{equation} 
Thus the mean curvature of the lightlike hypersurface is the expansion of leaves of the screen distribution in the outgoing direction. Hence, a leaf $S$ is a trapped submanifold if and only if $\stackrel{\star}{S}_1<0$ and $S_1>0$; $S$ is a marginally trapped submanifold if and only if $\stackrel{\star}{S}_1=0$ and $S_1\geq0$; $S$ is a trapped outer submanifold if and only if $\stackrel{\star}{S}_1<0$; $S$ is a marginally outer trapped submanifold if and only if $\stackrel{\star}{S}_1=0$; and $M$ is a null FOTH if and only if $\sS_1=0$, $S_1>0$ and $\delta_N\!\sS_1<0$. The mean curvature vector is given by
\begin{equation}\label{meancurv}
H=-tr(h)=-S_1\xi-\stackrel\star S_1N=-\theta^{(k)}\ell-\theta^{(\ell)}k.
\end{equation}

\begin{example}
	We consider the $6-$dimensional spacetime $(\Bm=\R^6,\bg)$ endowed with
	the metric
	$$\bar g=-(dx^0)^2+(dx^1)^2+\exp{2x^0}[(dx^2)^2+(dx^3)^2]+\exp{2x^1}[(dx^4)^2+(dx^5)^2],$$
	$(x^0,...,x^5)$ being the usual Cartesian coordinates on $\R^6$.
	The hypersurface $M$ of $\Bm$ defined by
	$$M=\{(x^0,...,x^5)\in\R^6\,;\;x^0+x^1=0\}$$
	is a lightlike hypersurface of $(\bar M, \bar g)$ and
	the vector field $N=-\frac{1}{2}\left(\frac{\partial}{\partial
		x^0}+\frac{\partial}{\partial x^1}\right)$ is a null rigging for
	$M$ with corresponding rigged vector field\\ $\xi=\frac{\partial}{\partial x^0}-
	\frac{\partial}{\partial x^1}$ and integrable screen distribution $
	S(TM)=span\{\stackrel{\star}{E}_1,
	\stackrel{\star}{E}_2,\stackrel{\star}{E}_3,\stackrel{\star}{E}_4\}$
	with
	$$\stackrel{\star}E_1=e^{-2x^0}\frac{\partial}{\partial x^2},~\stackrel{\star}E_2=
	e^{-2x^0}\frac{\partial}{\partial
		x^3},~\stackrel{\star}E_3=e^{-2x^1} \frac{\partial}{\partial
		x^4},~\stackrel{\star}E_4=e^{-2x^1}\frac{\partial}{\partial
		x^5}.$$ By direct computations, one sees that $M$ is not totally geodesic but is minimal (hence $\Bm$ doesn't have constant sectional curvature) and $S_N=tr(A_N)=2>0$. Hence leaves of the screen distribution given by $L=\{x^0=cste,x^1=cste\}$ are marginally trapped submanifolds of $(\Bm,\bg)$.
\end{example}

\begin{example}
	Let us consider the Schwarzschild spacetime, whose metric $\bg$ is given in the Schwarzschild coordinates $(t_s,r,\theta,\varphi)$ by
	\begin{equation}\label{schwarzschild}
	\bg=-\left(1-\frac{2m}{r}\right)dt_s^2+\left(1-\frac{2m}{r}\right)^{-1}dr^2+r^2(d\theta^2+\sin^2\theta d\varphi^2).
	\end{equation}
	This is the first non-trivial solution of Einstein's equations, found by the astrophysicist Karl Schwarzschild in the end of 1915. It is the metric outside a spherical body of mass $m$ and radius $r=2m$.  The singularity presented by this metric on $r=2m$ is an apparent singularity due to the bad choice of coordinate systems. Let us set
	$$t=t_s+2m\ln\left|\frac{r}{2m}-1\right|.$$
	Then, $(t,r,\theta,\varphi)$ is a new coordinate systems called ingoing Eddington-Finkelstein coordinates, and in which metric (\ref{schwarzschild}) becomes
	\begin{equation*}%\label{schwarzschild2}
	\bg=-\left(1-\frac{2m}{r}\right)dt^2+\frac{4m}{r}dtdr+\left(1+\frac{2m}{r}\right)dr^2+r^2(d\theta^2+\sin^2\theta d\varphi^2).
	\end{equation*}
	It is nothing to see that the hypersurface
	$$M:~r=2m,$$ is a  null hypersurface, called the event horizon. One can rig this lightlike hypersurface by the following future directed null rigging and associated rigged :
	$$ N=\frac{r}{2m}\partial_t-\frac{r}{2m}\partial_r,~~\xi=-\partial_t.$$ Corresponding screen distribution is given by $S(TM)=span(\sE_1,\sE_2)$ with 
	$$\sE_1=\partial_\theta,~~\sE_2=\partial_\varphi.$$ This screen is integrable and leaves are spheres $\{t=cste, r=2m\}$, which foliate $M$. A direct computation gives
	$$\sn=0,~~A_N=\frac{1}{2m}P,~~\tau|_{S(TM)}=0,~~\tau(\xi)=\frac{1}{4m}.$$
	Hence, $M$ is totally geodesic, the screen distribution is totally umbilical with $\lambda=1/2m$. It follows that 
	$$\sS_1=0,~~S_1=1/m\geq0.$$ Then, spheres $t=cste,r=2m$ are marginally trapped surfaces. Hence, $M$ is a null future trapping horizon. Also, $M$ is a non-expanding horizon and spheres $\{t=cste, r=2m\}$ are cross sections.
	
	Moreover, the one parameter group of diffeomorphisms (just the flow) of $N$ starting at $(t_0,r_0,\theta_0,\varphi_0)$ is given by
	$$\phi_\epsilon=(t_0+(1-\exp(-\epsilon/2m))r_0,r_0\exp(-\epsilon/2m),\theta_0,\varphi_0).$$
	The image of a sphere $\mathcal{S}=\{t=t_0,r=2m\}$ by $\phi_\epsilon$ (Lie dragging of $\mathcal{S}$ along $N$) is
	$$\mathcal{S}_\epsilon:=\phi_\epsilon(\mathcal{S})=\{t=t_0+2m(1-\exp(-\epsilon/2m)),r=2m\exp(-\epsilon/2m)\}.$$
	These spheres are spacelike surfaces and corresponding normalized pairs are given by
	$$N_\epsilon=\frac{r\exp(-\epsilon/2m)}{2m}(1,-1,0,0),$$ $$\xi_\epsilon=\left(-\frac{m\exp(\epsilon/2m)}{x}-\frac{2m^2\exp(\epsilon/m)}{x^2},-\frac{m\exp(\epsilon/2m)}{x}+\frac{2m^2\exp(\epsilon/m)}{x^2},0,0\right).$$
	By a direct calculation, one finds
	$$\sS_{\epsilon1}=\frac{2m\exp(\epsilon/2m)(-r+2m\exp(\epsilon/2m))}{r^3}.$$
	Hence,
	$$\delta_N\!\sS_1=\left.\frac{d\sS_{\epsilon1}}{d\epsilon}\right|_{\epsilon=0}=\frac{r-4m}{r^3}<0.$$
	$M$ is then a null FOTH (a quasi-local black hole horizon).
\end{example}

Let $\overline{R}, R, \stackrel{\star}{R}$ be the Riemannian curvatures of $\bnab$, $\nabla$ and $\snab$ respectively. It is a straightforward computation to check that the Codazzi and Ricci equations of the immersion $S\to\Bm$ are given by
\begin{align}
\langle\overline{R}(X,Y)Z,T\rangle&=\langle\stackrel{\star}{R}(X,Y)Z,T\rangle\nonumber\\
&+ \langle A_{h(X,Z)}Y,T\rangle-\langle A_{h(Y,Z)}X,T\rangle,\label{geqs1}\\
\langle\overline{R}(X,Y)Z,\mathsf{\delta_1}\rangle&=\langle\left(\bnab_Xh\right)(Y,Z),\delta_1 \rangle-\langle\left(\bnab_Yh\right)(X,Z),\delta_1\rangle,\label{geqs2}\\
\langle\overline{R}(X,Y)\delta_1,\delta_2\rangle&=\langle R^{\perp}(X,Y)\delta_1 ,\delta_2\rangle-\langle[A_{\delta_1},A_{\delta_2}]X,Y\rangle,\label{geqs3}
\end{align}
for all $X,Y,Z,T\in\Gamma(S(TM))$ and for any $\delta_1,\delta_2\in\Gamma(TS^\perp)$. Here, $\bnab h$ is defined by
$$(\bnab_Xh)(Y,Z)=\snab^\perp_X\!\!h(Y,Z)-h(\snab_X\!\!Y,Z)-h(Y,\snab_X\!\!Z).$$
It is straightforward to show that
\begin{equation}\label{nablah}
\bnab_Xh=\left(\nabla_XB+\tau(X)B\right)N+\left(\nabla_XC-\tau(X)C\right)\xi.
\end{equation}

With the objects of the null hypersurface $M$, the Codazzi and Ricci equations (\ref{geqs1})-(\ref{geqs3}) give
\begin{align}
\langle\overline{R}(X,Y)Z,T\rangle&=\langle R(X,Y)Z,T\rangle\nonumber\\\label{geqs4}&+B(X,Z)C(Y,T)-B(Y,Z)C(X,T)\\
&=\langle \stackrel{\star}{R}(X,Y)Z,T\rangle\nonumber\\\nonumber&+B(X,Z)C(Y,T)-B(Y,Z)C(X,T)\\\nonumber&+C(X,Z)B(Y,T)-C(Y,Z)B(X,T)\\
\langle\overline{R}(X,Y)Z,N\rangle&=\left(\nabla_XC\right)(Y,Z)-\left(\nabla_YC\right)(X,Z)\nonumber\\&+C(X,Z)\tau(Y)-C(Y,Z)\tau(X),\label{geqs5}\\
\langle\overline{R}(X,Y)Z,\xi\rangle&=\left(\nabla_XB\right)(Y,Z)-\left(\nabla_YB\right)(X,Z)\nonumber\\&+B(Y,Z)\tau(X)-B(X,Z)\tau(Y).\label{geqs6}\\
\langle\overline{R}(X,Y)\xi,N\rangle&=C(\sn\!\!X,Y)-C(X,\sn\!\!Y)-d\tau(X,Y).\label{geqs8}
\end{align}

We say that $S$ has a parallel mean curvature vector when $\snab^\perp\!\!\!H=0$. $S$ is called a parallel submanifold when $\bnab h=0$. One can check that parallel submanifolds have parallel mean curvature vector. 

\begin{proposition}\label{prominimal}
	Let $(M,g,N)$ be a closed rigged null hypersurface of a Lorentzian manifold $(\bg,\Bm)$, and $L$ a generic leaf of the screen distribution. Then,
	\begin{dingautolist}{192}
		\item $M$ is a null MOTT if and only if $M$ is minimal ;
		\item $L$ has parallel mean curvature vector if and only if expansions satisfy $d\theta^{(k)}=\theta^{(k)}\tau$ and $d\theta^{(\ell)}=-\theta^{(\ell)}\tau$ on $L$.
	\end{dingautolist}
\end{proposition}

\begin{proof}
	Taking normal derivative in (\ref{meancurv}), one obtains 
	\begin{equation}\label{normalderivaiveH}
	\snab^\perp\!\!\!\!H = (d\theta^{(k)}-\theta^{(k)}\tau)\xi-(d\theta^{(\ell)}+\theta^{(\ell)}\tau)N
	\end{equation}
	and the second item follows.
	If $M$ is minimal then by Definition \ref{minimal} and (\ref{expansion}), the outgoing expansion of $L$ identically vanishes and then, $M$ is a null MOTT. Conversely, if $M$ is a null MOTT then, $M$ is foliated by MOTSs. Since This MOTSs are spacelike, the distribution associated to this foliation is a screen distribution for $M$. Let $(N,\xi)$ be a normalizing pair associated with this screen distribution as given into Theorem \ref{db}. The second equality in (\ref{expansion}) give $tr(\sn)=\theta^{(\ell)}=0$, which by Definition \ref{minimal} leads to: $M$ is minimal.
\end{proof}

%\begin{theorem}
%Let $(\Bm,\bg)$ be an Einstein's space-time where null energy condition and dominant energy condition hold. Any MOTT foliated by closed MOTS is a trapping horizon.
%\end{theorem}

\begin{theorem}
	In a spacetime with constant sectional curvature $c$, cross-sections of a MOTT are Riemannian manifold with the same constant sectional curvature $c$.
\end{theorem}

\begin{proof}
	With Proposition \ref{prominimal}, such a MOTT $M$ is minimal and since the sectional curvature of the ambient is constant, by Theorem \ref{geomini}, $M$ is totally geodesic. So, equation (\ref{geqs4}) leads to
	\begin{equation}
	\langle\overline{R}(X,Y)Z,T\rangle=\langle \stackrel{\star}{R}(X,Y)Z,T\rangle
	\end{equation}
	which completes the proof.
\end{proof}

From the above theorem, one deduces following corollaries. The first one is justified by the fact that a Riemannian manifold with non-positive sectional curvature cannot be isometric to the sphere $\mathbb{S}^2$ while the second one holds as any compact Riemannian manifold is geodesically complete and then is a space form provided the sectional curvature is positive.

\begin{corollary}
	A space-time with constant non-positive sectional curvature cannot contain a null non-expanding horizon.
\end{corollary}

\begin{corollary}
	In an Einstein's space-time with positive constant sectional curvature and where dominant energy condition holds, any null trapping horizon is a null non-expanding horizon.
\end{corollary}

\subsection{SIC-normalized null hypersurfaces}

Let $(M,g,N)\to(\Bm,\bg)$ be a rigged null hypersurface of a $(n+2)-$dimensional spacetime. We say that the screen distribution is conformal when the two shape operators are conformal. This means that there exists a function $\varphi$ such that $A_N=\varphi\sn$.  

\begin{definition}
	A Screen Integrable and Conformal (SIC) rigging is one for which the screen distribution is integrable and conformal.
\end{definition}

\begin{example}
	The rigging $N=k$ corresponding to the screen distribution coming from the MOTSs foliation in a MOTT or a trapping horizon is a SIC-rigging. 
\end{example}

When the conformal factor $\varphi$ is $1$ and the rigging is closed, $N$ is called a UCC (Unitary Conformally Closed) rigging. Hence, UCC rigging defined in \cite{AFN}, is a particular case of SIC rigging. Since the screen distribution $S(TM)=ker(\eta)$ then from the two last items of Lemma \ref{change}, it follows that if $N$ is a SIC-rigging then any change of rigging $\widetilde N=\phi N$ is also a SIC-rigging. 

\begin{lemma}\label{cc}
	Let $x:(M,g,N)\to(\Bm,\langle\cdot,\cdot\rangle)$ be a rigged null hypersurface.
	\begin{dingautolist}{192}
		\item If $N$ is a closed rigging with conformal screen distribution then the rotation $1-$form $\tau^N$ vanishes on the screen distribution. 
		\item If $N$ is a rigging with conformal screen distribution and vanishing rotation $1-$form then, the $1-$form $\eta=x^\star\langle N,\cdot\rangle$ is closed.
	\end{dingautolist}
	%Moreover if $\Bm$ is a space form then, the rotation $1-$form identically vanishes.
\end{lemma}

\begin{proof}
	Assume $\eta$ is closed and let
	$U$, $V$ be tangent vector fields to $M$. The condition $d\eta(U,V)=U\cdot
	\eta(V) - V\cdot\eta(U) -\eta([U,V]) = 0$ is equivalent to
	$\Big\langle\overline\nabla_{U}N,V  \Big\rangle =
	\Big\langle\overline\nabla_{V}N,U  \Big\rangle$. Then by the weingarten
	formula, we get
	\begin{equation}\label{eq}
	\Big\langle -A_{N}U,V  \Big\rangle + \tau^{N}(U)\eta(V) = \Big\langle -A_{N}V,U  \Big\rangle + \tau^{N}(V)\eta(U).
	\end{equation}
	In this relation, take $V = \xi$ to get
	$$\tau^{N}(U) = -\Big\langle A_{N}\xi,U  \Big\rangle + \tau^{N}(\xi)\eta(U)$$
	% which gives the desired formula as $\tau^{N}(\xi) = 0$.
	From here if $N$ is with conformal screen then $A_N\xi=\varphi\!\!\sn\!\!\xi=0$ and the first item is proved. If $\tau^N$ identically vanishes and $N$ is with conformal screen distribution then $A_N$ is symmetric and equation (\ref{eq}) holds, that is equivalent to say that $\eta$ is closed.
\end{proof}

A consequence of the second item in the above Lemma is that rigging with conformal (resp. unitary conformally) screen distribution and vanishing rotation $1-$form is a SIC-rigging (resp. UCC-rigging). The following is to show that there exist null hypersurfaces with SIC-rigging. 

Let $M$ be the Monge hypersurface of the Lorentz-Minkowski space $\R^{n+2}_1$ given by 
$$M=\{x=(x^0=F(x^1,\ldots,x^{n+1}),x^1,\ldots,x^{n+1})\in\R^{n+2}\},$$ where  $F:D\rightarrow\R$ is a smooth function defined on some open subset  $D$ of $\R^{n+1}$. 
For
a vector field $\displaystyle  X = X^A\frac{\partial}{\partial
	x^A}\in\R^{n+2}_{1}$ a necessary and sufficient condition  to be
tangent to $M$  is that $X^0=X^1F'_{x^1}+\cdots+X^1F'_{x^{n+1}}$.
Then $\displaystyle \delta=\dfrac{\partial}{\partial
	x^0}+\displaystyle\sum_{a=1}^{n+1}F'_{x^a}\frac{\partial}{\partial
	x^a}$ is normal to $M$. The later  is a null hypersurface if and
only if $\delta$ is a  null vector field. This is equivalent to
\begin{equation}\label{lightlike}
\sum_{a=1}^{n+1}\left(F'_{x^a}\right)^2=||\nabla F||^{2}=1,
\end{equation}
where $\nabla F$ is the gradient of $F$ with respect to the
Euclidean structure  $||\cdot||$  of $\R^{n+1}$.

Let us assume that $M$ is a Monge null hypersurface and consider the null rigging
\begin{equation}\label{mongenormalization}
\mathscr{N}_{F}=\frac{1}{\sqrt2}\Big[-\frac{\partial}{\partial
	x^0}+\sum_{a=1}^{n+1}F'_{x^a}\frac{\partial}{\partial
	x^a}\Big]=\frac1{\sqrt2}(-1,\nabla F),
\end{equation} 
with corresponding rigged vector field
\begin{equation}\label{rigging3}
\xi_{F}=\frac{1}{\sqrt2}\Big[\frac{\partial}{\partial
	x^0}+\frac1{\sqrt2}\sum_{a=1}^{n+1}F'_{x^a}\frac{\partial}{\partial
	x^a}\Big]=\frac{1}{\sqrt2}(1,\nabla F).
\end{equation} 
Let us consider  the natural (global) parametrization of $M$ given by
\begin{equation}
\begin{cases} x^0=F(u^1,...,u^{n+1}) &\\
x^a=u^a \qquad\qquad\qquad\qquad & (u^1,...,u^{n+1})\in D\\
a=1,...,n+1
\end{cases}.
\end{equation}
Then $\Gamma(TM)$ is spanned by $\{\frac{\partial}{\partial
	u^a}\}_a$ with
\begin{equation}
\frac{\partial}{\partial u^a}=F'_{u^a}\frac{\partial}{\partial
	x^0}+\frac{\partial}{\partial x^a}.
\end{equation}

Then, taking
partial derivative of (\ref{lightlike}) with respect to $x^b$
($1\leq b\leq n+1)$ leads to
\begin{equation}\label{lightlike1}
\sum_{a=1}^{n+1}F'_{x^a}F''_{x^ax^b}=0.
\end{equation}
Now take a covariant derivative by the flat connection $\overline\nabla$ and
using (\ref{lightlike1}),
\begin{eqnarray*}
	\overline\nabla_\frac{\partial}{\partial
		u^a}\xi_F&=&\frac1{\sqrt2}\sum_{b=1}^{n+1}F''_{u^au^b}\frac{\partial}{\partial
		x^b}\\
	&=&\frac1{\sqrt2}\sum_{b=1}^{n+1}\left(-F''_{u^au^b}F'_{u^b}\frac{\partial}{\partial
		x^0}+F''_{u^au^b}\frac{\partial}{\partial
		x^b}\right)\\
	\overline\nabla_\frac{\partial}{\partial
		u^a}\xi_F&=&\frac1{\sqrt2}\sum_{b=1}^{n+1}F''_{u^au^b}\frac{\partial}{\partial{u^b}}=\overline\nabla_\frac{\partial}{\partial
		u^a}\mathscr{N}_F,
\end{eqnarray*}
which belong to $\Gamma(S(TM))$ (one proves this by using (\ref{lightlike1})) and shows that the rotation $1-$form identically vanishes and
\begin{equation*}
A_{\mathscr{N}_{F}}\left(\frac{\partial}{\partial
	u^a}\right)=\stackrel\star A_{\xi_F}\left(\frac{\partial}{\partial
	u^a}\right)=-\frac1{\sqrt2}\sum_{b=1}^{n+1}F''_{u^au^b}\frac{\partial}{\partial
	u^b}.
\end{equation*}
Thus, the screen distribution is conformal with  conformal
factor $\phi=1$. Hence by Lemma \ref{cc}, $\mathscr{N}_{F}$ is a UCC-rigging. $\mathscr{N}_{F}$ is the generic UCC-rigging of the Monge null hypersurface $M$. This proves the existence of so many null hypersurfaces of $\R^{n+2}_1$ equipped with a UCC-rigging. Rescaling  $\mathscr{N}_{F}$ by a nowhere zero smooth function $\phi$ to have $N=\phi\mathscr{N}_{F}$, it is easy to check that when $\phi$ is not constant on the screen distribution, $N$ is a SIC-rigging which is not a UCC-rigging.
The matrix of $\stackrel\star A_{\xi_F}$ with respect to the
basis $\{\frac{\partial}{\partial u^a}\}_a$ is given by
\begin{equation}\label{axi}
\stackrel\star A_{\xi_F}=-\frac1{\sqrt2}
\begin{pmatrix}
F''_{u^1u^1} & \cdots & F''_{u^{1}u^{n+1}}\\ \vdots & \ddots &
\vdots \\F''_{u^{n+1}u^1} & \cdots & F''_{u^{n+1}u^{n+1}}
\end{pmatrix}=-\frac1{\sqrt2}Hess(F)
\end{equation}
and by using (\ref{expansion}), it follows that expansions are given by,
\begin{equation}\label{s1}
-\theta_{N^+}=\theta_{\xi^+}=tr(\stackrel\star A_{\xi_F})=-\frac1{\sqrt2}\displaystyle\sum_{b=1}^{n+1}F''_{u^bu^b}=-\frac1{\sqrt2}\Delta
F,
\end{equation}
Which shows that expansions are of different signs. This is a general fact for all null hypersurfaces in the Lorentz-Minkowski space $\R^{n+2}_1$. Hence in Lorentz-Minkowski space, a null hypersurface cannot be foliated by trapped submanifolds. From the above equation, one derives the following result.

\begin{proposition}\label{promots}
	Let $x:(M,g,\mathscr{N}_F)\to\R^{n+2}_1$ be a  Monge null
	hypersurface endowed with its generic UCC-rigging $\mathscr{N}_F$. Then, the screen distribution is integrable with leaves the level sets of the function $F$, and a leaf $S$ is a Marginally (Outer) Trapped Submanifold if and only if $F$ is harmonic on $S$.
\end{proposition}

\begin{theorem}
	A Monge null hypersurface $M\to\R^{n+2}_1$ graph of a function $F$ is a MOTT if and only if $F$ is harmonic.
\end{theorem}

\begin{proof}
	If $M$ is a MOTT then, there exists a foliation of $M$ by MOTSs of $\R^{n+2}_1$. The distribution of this foliation can be set as well as a screen distribution on $M$. By Proposition \ref{prominimal} and Definition \ref{minimal}, $M$ is minimal and it follows from equality (\ref{s1}) that $F$ is harmonic. Conversely, if $F$ is harmonic then endowed $M$ with the generic UCC-rigging, it follows from Proposition \ref{promots} that the screen distribution is integrable and leaves are MOTSs. Meaning that $M$ is a MOTT.
\end{proof}

From now on, $(\Bm(c),\bg)$ is an $(n+2)-$dimensional Lorentzian manifold with constant sectional curvature $c\in\R$, and $x:(M,g,N)\to(\Bm(c),\bg)$ is a SIC-rigged null hypersurface. From here, Gauss-Codazzi equations (\ref{geqs4})-(\ref{geqs8}) become

\begin{align}
\langle R(X,Y)Z,T\rangle&=c\left(\langle Y,Z\rangle\langle X,T\rangle-\langle X,Z\rangle\langle Y,T\rangle\right)\nonumber\\&+\varphi\left(B(Y,Z)B(X,T)-B(X,Z)B(Y,T)\right),\label{geqs13}\\
\langle \stackrel{\star}{R}(X,Y)Z,T\rangle&=c\left(\langle Y,Z\rangle\langle X,T\rangle-\langle X,Z\rangle\langle Y,T\rangle\right)\nonumber\\&+2\varphi\left(B(Y,Z)B(X,T)-B(X,Z)B(Y,T)\right),\label{geqs14}
\end{align}\vspace{-.9cm}
\begin{equation}
\left(\nabla_XB\right)(Y,Z)+B(Y,Z)\tau(X)=\left(\nabla_YB\right)(X,Z)+B(X,Z)\tau(Y),\label{geqs15}
\end{equation}
for all $X,Y,Z,T\in\Gamma(S(TM))$. From equations (\ref{geqs13}) and (\ref{geqs14}) above, the following result is straightforward. 

\begin{proposition}
	Let $x:(M,g,N)\to(\Bm,\bg)$ be a SIC-rigged null hypersurface of a locally flat Lorentzian manifold (thus $c=0$). Then, $(M,\nabla)$ is locally flat if and only if $(S, \snab)$ is locally flat, for any leaf $S$ of the screen distribution.
\end{proposition}

\begin{theorem}\label{lemma}
	Let $x:(M,g,N)\to(\Bm(c),\bg)$ be a UCC-rigged null hypersurface, and $S$ a generic leaf of the screen distribution. Then,
	\begin{dingautolist}{192}
		\item $S$ is parallel if and only if $B$ is parallel;
		\item $S$ has parallel mean curvature if and only if the two expansions $\theta_N$ and $\theta_\xi$ are closed;
		\item if $M$ is totally umbilical with $B=\rho g$, then $\rho$ is constant on each (connected) leaf of the screen distribution. In addition if $\tau(\xi)=\rho$ then $\rho$ is constant on $M$.
		%\item if $M$ is proper totally umbilical, then $M$ can't be locally flat;
		%\item if $M$ is totally geodesic, then $M$ is locally flat.
	\end{dingautolist}
\end{theorem}

\begin{proof}
	Notice that the rotation $1-$form $\tau$ vanishes on the screen distribution since the rigging is UCC (Lemma \ref{cc}). Notice also that $\varphi=1$ and $B=C$. Hence, the first (resp. second) item follows from the equality (\ref{nablah}) (resp. (\ref{normalderivaiveH})). From equation (\ref{geqs15}), one has
	\begin{equation*}
	(X\cdot \rho)g(Y,Z)+\rho(\nabla_Xg)(Y,Z)=(Y\cdot \rho)g(X,Z)+\rho(\nabla_Yg)(X,Z).
	\end{equation*}
	Since $\nabla$ is $g-$metric on sections of the screen distribution (Proof: equation (\ref{met})), the latter equation leads to $(X\cdot\rho)Y=(Y\cdot\rho)X$, for all $X,Y\in\Gamma(S(TM))$. This show that $(X\cdot\rho)=0$, for all $X\in\Gamma(S(TM))$. This completes the proof of item \ding{194}.
\end{proof}

%\begin{proposition}
%Let $x:(M,g,N)\to(\Bm,\bg)$ be a UCC-rigged null hypersurface, and $S$ a generic leaf of the screen distribution. Then,
%\begin{dingautolist}{192}
%\item $S$ is parallel if and only if $B$ is parallel ;
%\item $S$ has parallel mean curvature if and only if the two expansions $\theta_N$ and $\theta_\xi$ are closed.
%\end{dingautolist}
%\end{proposition}

\begin{theorem}
	Let $x:(M,g,N)\to(\Bm(c),\bg)$ be a SIC-rigged null hypersurface of a Lorentzian manifold with constant sectional curvature $c$. If $M$ is totally umbilical (or geodesic) then, each (connected) leaf of the screen distribution is a space form.
	% and has parallel mean curvature vector.
\end{theorem}

\begin{proof}
	Assume that $M$ is totally umbilical with $B=\rho g$. Then equation (\ref{geqs14}) gives 
	\begin{equation*}
	\langle\stackrel{\star}{R}(X,Y)Z,T\rangle=(c+2\varphi\rho^2)\left(\langle Y,Z\rangle\langle X,T\rangle-\langle X,Z\rangle\langle Y,T\rangle\right),
	\end{equation*}
	Which show that sectional curvature of $S$ is constant $\kappa=c+2\varphi\rho^2$ for all sections $\sigma=span(X,Y)$. It is then known that $\kappa$ is a constant function.
	This prove that each connected leaf of the screen distribution has constant sectional curvature $$\kappa=c+2\varphi\rho^2.$$ 
\end{proof}

Hence for a SIC-rigged totally umbilical null hypersurface $(M,g,N)\to(\Bm(c),\bg)$ with $B=\rho g$ and $A_N=\varphi\sn$, the product $\varphi\rho^2$ is constant on each leaf of the screen distribution. It is noteworthy that a closed rigging with conformal screen distribution is a SIC-rigging.

\begin{corollary}
	Let $x:(M,g,N)\to(\Bm(c),\bg)$ be a  null hypersurface endowed with a closed rigging with conformal screen  distribution with $A_N=\varphi\sn$. If $M$ is totally umbilical with $B=\rho g$ then, $\rho$ and $\varphi$ are constants on each (connected) leaf of the screen distribution. 
\end{corollary}

\begin{proof}
	By Lemma \ref{cc}, the rotation $1-$form vanishes on the screen distribution and equation (\ref{geqs15}) become
	\begin{equation*}
	\left(\nabla_XB\right)(Y,Z)=\left(\nabla_YB\right)(X,Z).
	\end{equation*}
	Now, if $M$ is totally umbilical then by using the above equation, one shows that $\rho$ is constant on each leaf of the screen distribution. It follows that $\varphi$ also is constant on each leaf of the screen distribution, since it is the case for $\varphi\rho^2$.
\end{proof}

The following is a direct consequence of the above Theorem.
\begin{corollary}
	Let $x:(M,g,N)\to\R^{n+2}_1$ be a UCC-rigged null hypersurface. If $M$ is totally umbilical with $B=\rho g$ then, each (connected) leaf of the screen distribution is a Riemannian manifold with positive constant sectional curvature $c=2\rho^2$ and parallel mean curvature vector.
\end{corollary}

%\begin{theorem}
%Let $x:(M,g,N)\to\R^{n+2}_1$ be a UCC-rigged null hypersurface. If $M$ is proper totally umbilical then, any leaf of the screen distribution can't be marginally trapped.
%\end{theorem}

\begin{example}
	Let $x:\Lambda_0^{n+1}\to\R^{n+2}_1$, $p=(x^1,\ldots x^{n+1})\mapsto x=(x^0=F(x^1,\ldots x^{n+1}),x^1,\ldots x^{n+1})$ be the future null cone, which is the graph of the function
	$$F=\left(\sum_{a=1}^{n+1}(x^a)^2\right)^{1/2}.$$
	This is  a totally umbilical null hypersurface in  $\R^{n+2}_1$ and  the generic
	$UCC$-rigging (\ref{mongenormalization}) becomes
	$$\mathscr{N}_{F}=-\frac1{\sqrt2}\frac{\partial}{\partial x^0}+\frac{1}{x^0\sqrt2}\sum_{a=1}^{n+1}(x^a)\frac{\partial}{\partial
		x^a},$$ and the corresponding rigged vector field
	$$\xi_{F}=-\frac1{\sqrt2}\frac{\partial}{\partial x^0}-\frac{1}{x^0\sqrt2}\sum_{a=1}^{n+1}(x^a)\frac{\partial}{\partial x^a}.$$ For this rigging, the screen distribution is integrable and leaves of the screen distribution are sections of the future lightcone by hyperplanes $x^0=cste$. These are spheres of radius $x^0$ centered at $(x^0,0,\ldots,0)\in\R^{n+2}$. (See \cite{AFN} for a proof.)  All the principal curvatures are given by
	$$\rho=\frac{1}{x^0\sqrt2},$$
	which are constants on each leaf of the screen distribution, and this agrees with Theorem \ref{lemma}.
	By the Theorem above, each leaf of the screen distribution has positive constant sectional curvature $c=2\rho^2=\frac{1}{(x^0)^2}$, which is really the sectional curvature of a sphere of radius $x^0$.
\end{example}

\bibliographystyle{amsalpha}
{\footnotesize\bibliography{NullHypersurfacesTrappingHorizons}}

\providecommand{\bysame}{\leavevmode\hbox to3em{\hrulefill}\thinspace}
\providecommand{\MR}{\relax\ifhmode\unskip\space\fi MR }
\providecommand{\MRhref}[2]{%
  \href{http://www.ams.org/mathscinet-getitem?mr=#1}{#2}
}
\providecommand{\href}[2]{#2}
\begin{thebibliography}{AEPVDB04}

\bibitem[AEPVDB04]{ashtekar2004multipole}
Abhay Ashtekar, Jonathan Engle, Tomasz Pawlowski, and Chris Van Den~Broeck,
  \emph{Multipole moments of isolated horizons}, Classical and Quantum Gravity
  \textbf{21} (2004), no.~11, 25--49.

\bibitem[AF15]{AF}
C.~Atindogbe and H.~T. Fotsing, \emph{Newton transformations on null
  hypersurfaces}, Communication in Mathematics \textbf{23} (2015), no.~1,
  57--83,
  \href{http://dml.cz/handle/10338.dmlcz/144359}{http://dml.cz/handle/10338.dmlcz/144359}.

\bibitem[AFN16]{AFN}
C.~Atindogbe, H.~T. Fotsing, and F.~Ngakeu, \emph{Normalized null hypersurfaces
  in the lorentz-minkowski space satisfying $l_r\psi=u\psi+b$}, Submited
  (2016).

\bibitem[AK04]{ashtekar2004isolated}
Abhay Ashtekar and Badri Krishnan, \emph{Isolated and dynamical horizons and
  their applications}, Living Reviews in Relativity \textbf{7} (2004), no.~1,
  10.

\bibitem[Ala13]{alan2013black}
Hayward~Sean Alan, \emph{Black holes: new horizons}, World Scientific, 2013.

\bibitem[AMS07]{andersson2007stability}
Lars Andersson, Marc Mars, and Walter Simon, \emph{Stability of marginally
  outer trapped surfaces and existence of marginally outer trapped tubes},
  arXiv preprint arXiv:0704.2889 (2007).

\bibitem[Ati10]{C3}
C.~Atindogbe, \emph{Blaschke type normalization on light-like hypersurfaces},
  Journal of Geometry and Physics \textbf{6(4)} (2010), 362--382.

\bibitem[Com06]{compere2006introduction}
Geoffrey Compere, \emph{An introduction to the mechanics of black holes}, arXiv
  preprint gr-qc/0611129 (2006).

\bibitem[Dam79]{Da}
T.~Damour, \emph{Quelques propri\'et\'es m\'ecaniques, électromagnétiques,
  thermodynamiques et quantiques des trous noirs}, Universit\'e Paris 6,
  Th\`ese de doctorat d'\'Etat, 1979.

\bibitem[Dam82]{Da2}
\bysame, \emph{Surface effects in black hole physics}, Proceedings of the
  Second Marcel Grossmann Meeting on General Relativity, North-Holland,
  Amsterdam, (1982), 587.

\bibitem[DB96]{DB}
K.~L. Duggal and A.~Bejancu, \emph{lightlike submanifolds of semi-riemannian
  manifolds and application}, Dordrecht : Springer Netherlands, 1996,
  \href{http://www.worldcat.org/title/lightlike-submanifolds-of-semi-riemannian-manifolds-and-applications/oclc/851375811}{ISBN
  9789401720892 9401720894}.

\bibitem[DP64]{DP}
R.~H. Dicke and P.~J.~E. Peebles, \emph{Evolution of the solar system and the
  expansion of the universe}, Phys. Rev. Letters \textbf{12} (1964), no.~15,
  435--437, \href{https://doi.org/10.1103/PhysRevLett.12.435}{DOI:
  10.1103/PhysRevLett.12.435}.

\bibitem[Gal00]{gregory2000}
Gregory~J. Galloway, \emph{Maximum principles for null hypersurfaces and null
  splitting theorems}, Annales Henri Poincaré \textbf{1} (2000), no.~3,
  543--567,
  \href{https://arxiv.org/pdf/math/9909158.pdf}{https://arxiv.org/pdf/math/9909158.pdf}.

\bibitem[GJ06]{GJ}
E.~Gourgoulhon and J.~L. Jaramillo, \emph{A $3+1$ perspective on null
  hypersurfaces and isolated horizons}, Physics Reports \textbf{423} (2006),
  159--294, \href{http://dx.doi.org/10.1016/j.physrep.2005.10.005}{DOI:
  10.1016/j.physrep.2005.10.005}.

\bibitem[GO15]{GO}
M.~Gutierrez and B.~Olea, \emph{Induced riemannian structures on null
  hypersurfaces}, Math. Nachr. \textbf{289} (2015), 1219--1236,
  \href{https://doi.org/10.1002/mana.201400355}{DOI: 10.1002/mana.201400355}.

\bibitem[Gou16]{Gou}
E.~Gourgoulhon, \emph{Geometry and physics of black holes}, Institut
  d'Astrophysique de Paris, France, 2016,
  \href{http://luth.obspm.fr/~luthier/gourgoulhon/bh16/}{http://luth.obspm.fr/~luthier/gourgoulhon/bh16/}.

\bibitem[H\'73]{Ha}
P.~H\'a\'jicek, \emph{Exact models of charged black holes. i. geometry of
  totally geodesic null hypersurface}, Commun. Math. Phys. \textbf{34} (1973),
  37.

\bibitem[H\'75]{Ha3}
\bysame, \emph{Stationary electrovacuum spacetimes with bifurcate horizons}, J.
  Math. Phys. \textbf{16} (1975), 5--18.

\bibitem[Hay94]{hayward1994general}
Sean~A Hayward, \emph{General laws of black-hole dynamics}, Physical Review D
  \textbf{49} (1994), no.~12, 6467.

\bibitem[Hay00]{hayward2000black}
\bysame, \emph{Black holes: New horizons}, arXiv preprint gr-qc/0008071 (2000).

\bibitem[HE73]{hawking1973large}
Stephen~W Hawking and George Francis~Rayner Ellis, \emph{The large scale
  structure of space-time}, vol.~1, Cambridge university press, 1973.

\bibitem[Jar11]{Ja}
J.~L. Jaramillo, \emph{An introduction to local black hole horizons in the
  $3+1$ approach to general relativity}, International Journal of Modern
  Physics \textbf{20(11)} (2011), 2169--2204,
  \href{http://dx.doi.org/10.1142/S0218271811020366}{DOI:
  10.1142/S0218271811020366}.

\bibitem[Jar12]{Ja2}
\bysame, \emph{Area inequalities for stable marginally trapped surfaces}, a
  \textbf{26} (2012), \href{http://dx.doi.org/
  10.1007/978-1-4614-4897-6_5}{DOI: 10.1007/978-1-4614-4897-6\_5}.

\bibitem[Kri14]{krishnan2014quasi}
Badri Krishnan, \emph{Quasi-local black hole horizons}, Springer Handbook of
  Spacetime, Springer, 2014, pp.~527--555.

\bibitem[LMS05]{LMW}
A.~Lars, M.~Marc, and S.~Walter Simon, \emph{Local existence of dynamical and
  trapping horizons}, Physical review letters \textbf{95} (2005), 111102,
  \href{http://10.1103/PhysRevLett.95.111102}{DOI:
  10.1103/PhysRevLett.95.111102}.

\bibitem[Mar14]{mars2014stability}
Marc Mars, \emph{Stability of marginally outer trapped surfaces and geometric
  inequalities}, General Relativity, Cosmology and Astrophysics, Springer,
  2014, pp.~191--208.

\bibitem[Pen65]{PR}
R.~Penrose, \emph{Gravitational collapse and space-time singularities}, Phys.
  Rev. Letters \textbf{14} (1965), 57--59.

\end{thebibliography}

\end{document}